\documentclass[11pt,reqno]{amsart}

\usepackage{fullpage}

\usepackage{setspace}
\usepackage{color}
\usepackage{hyperref}  
\usepackage{graphicx}
\usepackage{datetime}
\usepackage{verbatim}

\usepackage{srcltx}
\usepackage{enumerate}

\usepackage{amsfonts}
\usepackage{amsmath}
\numberwithin{equation}{section}

\newtheorem*{acknowledge}{Acknowledgement}

\def\EE{\mathbb E}

\def\cK{\mathcal K}
\def\cL{\mathcal S}
\def\cT{\mathcal T}
\def\cM{\mathcal M} 
\def\cL{\mathcal L}

\def\cT{\mathcal T}

\def\({\left(}
\def\){\right)}  
\def\<{\langle}
\def\>{\rangle}

\let\epsilon\varepsilon
\def\eps*{\epsilon^*}
\let\log\log

\def\({\left(}
\def\){\right)}
\def\:{\colon}
\def\[{\left[}
\def\]{\right]}

\def\cF{\mathcal F}

\def\Rodl{R\"odl}

\def\Rucinski{Ruci\'{n}ski}

\newtheorem{theorem}{Theorem} 
\newtheorem{corollary} [theorem]{Corollary}
\newtheorem{definition} [theorem]{Definition}
\newtheorem{problem} [theorem]{Problem}

\newtheorem{remark} [theorem]{Remark}
\newtheorem{lemma} [theorem]{Lemma}

\onehalfspace

\title{Universality of random graphs for graphs of maximum degree two}

\author[J.~H.~Kim]{Jeong Han Kim}
\address{School of Computational Sciences \\ Korea Institute for Advanced Study (KIAS) \\ Seoul, South Korea
{\rm(J.~H.~Kim)}} 
\email{jhkim@kias.re.kr}
\thanks{The first author was supported by the National Research Foundation of Korea (NRF) Grant funded by the Korean Government (MSIP) (NRF-2012R1A2A2A01018585) and KIAS internal Research Fund CG046001.}

\author[S.~J.~Lee]{Sang June Lee}
\address{Department of Mathematical Sciences \\  Korea Advanced Institute of Science and Technology (KAIST) \\ Daejeon, South Korea
{\rm(S.~J.~Lee)}} 
\email{sjlee242@gmail.com}
\thanks{The second author was supported by the National Research Foundation of Korea (NRF) Grant funded by the Korea Government (MSIP)(No. 2013042157). This work was partially done while the second author was visiting Korea Institute for Advanced Study (KIAS)}

\date{October 22, 2013}

\begin{document}

\shortdate
\settimeformat{ampmtime}

\begin{abstract}
For a family $\cF$ of graphs, a graph $G$ is called \emph{$\cF$-universal} if $G$ contains every graph in $\cF$ as a subgraph. Let $\cF_n(d)$ be the family of all graphs on $n$ vertices with maximum degree at most $d$. Dellamonica, Kohayakawa, \Rodl~and \Rucinski~\cite{Dellamonica13:universal} showed that, for $d\geq 3$, the random graph $G(n,p)$ is  $\cF_n(d)$-universal with high probability provided  $p\geq C\big(\frac{\log n}{n}\big)^{1/d}$ for a sufficiently large constant $C=C(d)$.
 In this paper we prove the missing part of the result, that is, the random graph $G(n,p)$ is  $\cF_n(2)$-universal with high probability provided 
 $p\geq C\big(\frac{\log n}{n}\big)^{1/2}$ for a sufficiently large constant~$C$.
\end{abstract}

\maketitle

\section{Introduction}

For a positive integer $n$ and a real number $p$ in the range $0\leq p\leq 1$, the random graph $G(n,p)$ on a set $V$ of $n$ elements may be
 obtained from the complete graph on $V$ by choosing each edge with probability $p$, independently of all other edges.

 After the random graph $G(n,p)$ was first introduced by Erd\H{o}s~\cite{Erdos1947} in 1947, the theory of the random graph has become an active area of research. One of the most interesting problems is the containment problem, in which one tries to obtain conditions on $p$ for the property that $G(n,p)$ contains a given graph $H$ as a subgraph with high probability. For example, when $n$ is even and the given graph $H$ is a perfect matching on $V$, then it is easy to see that $np-\log n\rightarrow \infty$ is a necessary condition, as there is an isolated vertex with substantial probability if $np-\log n$ is bounded. 
  Erd\H{o}s and R\'enyi~\cite{Erdos1966} showed the condition is also sufficient.
 In the case that $H$ is a Hamiltonian cycle, 
  Koml\'os and Szemer\'edi~\cite{Komlos1983} and Korshunov~\cite{Korshunov1977}
discovered that the easy necessary condition $np-\log n-\log\log n\rightarrow \infty$ is also sufficient. More generally, if $H$ is a factor of a strictly balanced graph, including a triangle, a cycle or a complete graph, Johansson, Kahn and Vu~\cite{Johansson2008} determined a necessary and sufficient condition for the containment problem with respect to $H$. For more information, the reader is referred to Bollob\'as~\cite{Bollobas2001} and Janson, {\L}uczak and \Rucinski~\cite{Janson00} and the references therein.

 One may also consider a family $\cF$ of graphs rather than a single graph $H$. 
For a family $\cF$ of graphs, a graph $G$ is called \emph{$\cF$-universal} if $G$ contains every graph in $\cF$ as a subgraph. 
There is extensive research on $\cF$-universal graphs 
when $\cF$ are families of trees~\cite{Chung1978, Chung1978-2}, spanning trees~\cite{  Bhatt1989, Chung1979, Chung1983, Friedman1987}, planar graphs of bounded degree~\cite{Bhatt1989}, graphs of bounded size~\cite{Babai1982, Rodl1981}, graphs of bounded degree~\cite{Alon2007-2, Alon08, Alon2007-3, Alon01, Capalbo1999}, and spanning graphs of bounded degree~\cite{Alon2002, Johannsen2013}, etc.

Since an $\cF$-universal graph $G$ must have the maximum degree greater than or equal to the maximum degrees of all graphs in $\cF$,  a family $\cF$ of graphs of bounded degree may be considered. For example, one may consider the family $\cT_{n}(d)$ of all trees on $n$ vertices with maximum degree at most $d$. Bhatt, Chung, Leighton and Rosenberg~\cite{Bhatt1989} showed that
 there is a $\cT_n(d)$-universal graph on $n$ vertices with maximum degree depending only on $d$.
For $d\geq \log n$, Johannsen, Krivelevich and Samotij~\cite{Johannsen2013} proved that there is a positive constant $c$ such that $G(n,p)$ with $p\geq cdn^{-1/3}\log n$ is asymptotically almost surely (a.a.s.) $\cT_n(d)$-universal; where, in general, a property holds \emph{asymptotically almost surely}, or simply \emph{a.a.s.}, if it holds with probability tending to $1$ as $n\rightarrow \infty$.  In particular, we have that $G(n,p)$ with $p\geq cn^{-1/3}(\log n)^2$ is a.a.s. $\cT_n(\log n)$-universal, and hence, $\cT_n(d)$-universal for a constant $d$.
For the family  $\cT_{(1-\epsilon)n}(d)$ of all trees on $(1-\epsilon)n$ vertices with maximum degree at most $d$, Alon, Krivelevich and Sudakov~\cite{Alon2007} showed that for every positive constant $\epsilon$ and positive integer $d$, there exists a constant $c=c(\epsilon, d)$ such that $G(n,p)$ with $p=c/n$ is a.a.s. $\cT_{(1-\epsilon)n}(d)$-universal. For more related results,~\cite{Dellamonica2008_2, Balogh2010, Balogh2011} may be referred.

In this paper, we consider the family $\cF_n(d)$ of all graphs on $n$ vertices with maximum degree at most $d$. For an even $n$ and $d=1$,  the $\cF_n(1)$-universality is equivalent to the containment problem for a perfect matching.  Provided that $n$ is divisible by $d+1$, one may easily see that $p\geq  \Big(\frac{(\log n)^{1/d}}{n}\Big)^{2/(d+1)}$ is a necessary condition for $G(n,p)$ being a.a.s. $\cF_n(d)$-universal, since $\cF_n(d)$ contains a $K_{d+1}$-factor, and hence,
every vertex must be contained in a copy of $K_{d+1}$.
On the other hand, Dellamonica, Kohayakawa, \Rodl~and \Rucinski~\cite{Dellamonica2008, Dellamonica2012} 
proved that $p\geq C\big(\frac{\log n}{n}\big)^{1/(2d)}$ is sufficient, for a sufficiently large constant $C$.

\begin{theorem}[\cite{Dellamonica2008, Dellamonica2012}]

For every integer $d\geq 2$, there exists a positive constant $C=C(d)$ such that if $p\geq C\Big(\frac{(\log n)^2}{n}\Big)^{1/(2d)}$, then the random graph $G(n,p)$ is a.a.s. $\cF_n(d)$-universal.
 \end{theorem}

Recently the above result was notably improved for $d\geq 3$.

 \begin{theorem}[\cite{Dellamonica13:universal, Dellamonica12:universal}]\label{thm:universality}  For every integer $d\geq 3$, there exists a positive constant $C=C(d)$ such that if $p\geq C\big(\frac{\log n}{n}\big)^{1/d}$, then the random graph $G(n,p)$ is a.a.s. $\cF_n(d)$-universal.
 \end{theorem}

 In this paper, we show that the statement of  Theorem~\ref{thm:universality} holds for $d=2$.

 \begin{theorem}\label{thm:main} There exists a positive constant $C$ such that if $p\geq C\big(\frac{\log n}{n}\big)^{1/2}$, then the random graph $G(n,p)$ is a.a.s. $\cF_n(2)$-universal.
 \end{theorem}

The rest of this paper is organized as follows. In the next section we define a notion of a `\emph{good}' graph and introduce two main lemmas, which imply Theorem~\ref{thm:main}. Sections 3 and 4 are for the proofs of the two lemmas.

\vskip 1em
In this paper, we will use the following notation and convention. \\
 \textbf{Notation and convention:} 
For a graph $G$ and $v\in V(G)$, the set $N_G(v)$ denotes the set of neighbors of $v$ in $G$. Similarly, for $U\subset V(G)$, the set $N_G(U)$ denotes the set of vertices which are adjacent to a vertex in $U$. The graph $G[U]$ denotes the subgraph of $G$ induced on $U$. For simplicity, we omit ÔfloorÕ and
ÔceilingÕ symbols when they are not essential.

\section{Good graph and two lemmas} 

In order to show Theorem~\ref{thm:main}, by monotonicity, it suffices to show the statement of Theorem~\ref{thm:main} with $p=C\big(\frac{\log n}{n}\big)^{1/2}$ for a sufficiently large constant $C$. Hence, from now on, we fix $p$ as $p=C\big(\frac{\log n}{n}\big)^{1/2}$.
Throughout the paper, we let
$$\mbox{$\delta=0.01$ \hskip 1em and \hskip 1em $\epsilon=0.001$.}$$

Now we provide the definition of a `\emph{good}' graph. Let $V$ be a vertex set on $n$ vertices. We fix a partition $V=V_0\cup V_1 \cup \cdots \cup V_6$ such that 
$$\mbox{$|V_1|=\cdots=|V_6|=\epsilon n$ \hskip 1em and \hskip 1em $|V_0|=(1-6\epsilon)n\geq (3/4)n$.}$$

For a graph $G$ on $V$ and $k=1 \mbox{ or }2$, let $U\subset V$ and $\cL$ be a collection of pairwise disjoint $k$-subsets of $V\setminus U$. We consider a bipartite graph $B(\cL, U)$ between $\cL$ and $U$, in which $L\in\cL$ and $u\in U$ are adjacent if and only if $L\subset N_G(u)$.

Now we are ready to define a good graph.
 \begin{definition}\label{def:good} 
 A graph $G$ on $V$ is called `\emph{$(n,C)$-good}' if the following properties hold.
 
\begin{enumerate}[$(P1)$]

\item
There exists  a matching $\cM$ of $G[V_0]$ with $|\cM|=2\epsilon n$ such that for all $U\subset V\setminus V(\cM)$ with $\displaystyle |U|\leq  \frac{\delta n}{C^2\log n}$, 
we have $$\Big|\Big\{ \{a,b\}\in \cM \; \big| \; a\sim u, b\sim u \mbox{ for some } u\in U\Big\}\Big|\geq  \frac{C^2\log n}{16 n} |\cM| |U|.$$

\item

 Let $k=1 \mbox{ or } 2$, and $\cL$ be a collection of pairwise disjoint $k$-subsets of $V$. 
 
 If $\displaystyle |\cL|\leq \frac{\delta}{C^k}\Big(\frac{n}{\log n}\Big)^{k/2}$, then, for $V_i$ with $V_i\cap\Big(\bigcup_{L\in \cL}L\Big)=\emptyset$, $i=1,..., 6$, we have that
\begin{equation}\label{eq:P2_1}
|N_{_{\! B(\cL,V_i)}}(\cL)|\geq (1-\delta)C^k\Big(\frac{\log n}{n}\Big)^{k/2}|\cL| |V_i|.
\end{equation}

 If $\displaystyle |\cL|\geq \frac{\log n}{C^{k-1}}\Big(\frac{n}{\log n}\Big)^{k/2}$, then, for all $U$ with $\displaystyle |U|\geq \frac{\log n}{C^{k-1}}\Big(\frac{n}{\log n}\Big)^{k/2}$ and $U\cap\Big(\bigcup_{L\in \cL}L\Big)=\emptyset$, 
the graph $B(\cL,U)$ has at least one edge.

\vskip 1em
\noindent \Big(No requirement is needed when $\displaystyle\frac{\delta}{C^k}\Big(\frac{n}{\log n}\Big)^{k/2}<|\cL|<\frac{\log n}{C^{k-1}}\Big(\frac{n}{\log n}\Big)^{k/2}$\Big).

\end{enumerate}

 \end{definition}
 
 \begin{remark}
For $p=C\big(\frac{\log n}{n}\big)^{1/2}$,  the above inequality~\eqref{eq:P2_1} may be written as 
$$|N_{_{\! B(\cL,V_i)}}(\cL)|\geq (1-\delta)p^k|\cL| |V_i|.$$
Notice that $p^k |\cL| |V_i|$ is the expected number of edges in $B(\cL,V_i)$ if $G$ were $G(n,p)$. It is easy to see that only few vertices of $V_i$ are of degree $2$ or more in $B(\cL,V_i)$. Hence, $|N_{B(\cL,V_i)}(\cL)|$ is almost the same as the number of edges in $B(\cL,V_i)$.

\end{remark}

 We will show the following two lemmas.
 \begin{lemma}\label{lem:good_universal} There exists a positive constant $C$ such that
an $(n,C)$-good graph is $\cF_n(2)$-universal provided that $n$ is sufficiently large.
 \end{lemma}

 \begin{lemma}\label{lem:good}
There exists a positive constant $C$ such that the random graph $G(n,p)$ on $V$ with $p=C\big(\frac{\log n}{n}\big)^{1/2}$ is a.a.s. $(n,C)$-good.
 \end{lemma}
 
 Our proof of Lemmas~\ref{lem:good_universal} and~\ref{lem:good} will be given in Sections~\ref{sec:good_universal} and~\ref{sec:good}, respectively.
 Theorem~\ref{thm:main} clearly follows from Lemmas~\ref{lem:good_universal} and~\ref{lem:good}.

\section{Universality of good graph} \label{sec:good_universal}

For the proof of Lemma~\ref{lem:good_universal}, we may assume that $H$ is a maximal graph in $\cF_n(2)$, in the sense that no edge may be added to $H$ to be a graph in $\cF_n(2)$. Then, it is easy to see that all but at most one vertex of $H$ have degree $2$. We will show that there exists a positive constant $C$ such that, for a sufficiently large $n$, an $(n,C)$-good graph $G$ contains a copy of $H$ as a subgraph. To this end, a partition of $W:=V(H)$ will be used, and each part will be embedded at a time.
A subset of $W$ is called \emph{$k$-independent} in $H$ if the distance between every distinct pair of vertices in the subset is greater than $k$. 

\begin{lemma}\label{clm:partition of X} Let $H$ be a maximal graph in $\cF_n(2)$. Then there is a partition $W:=V(H)=W_0\cup W_1\cup \cdots \cup W_6$ with
\begin{equation}\label{eq:size of W_i} |W_0|= 4\epsilon n, \hskip 1em |W_6|=2\epsilon n, \hskip 1em |W_i|\geq 2\epsilon n, \hskip 0.5em \mbox{$i=1,2,..., 5,$}\end{equation}
such that 
\begin{enumerate}
\item\label{item:cond3} $W_1,..., W_5$ are $2$-independent.
\item\label{item:cond1}  $W_6$ is $3$-independent, and all vertices of $W_6$ are of degree $2$. 
\item\label{item:cond2} $W_0=N_H(W_6)$.

\end{enumerate}

\end{lemma}

\begin{proof}
We first construct $W_6$ and $W_0$. Since the maximum degree of $H$ is $2$,
 for each vertex $v$ in $H$, there are at most $6$ vertices that are of distance $3$ or less from $v$, excluding $v$ itself.
  By the greedy algorithm, it is easy to see that there is a 3-independent set of size at least $n/7$. Hence, we may choose $W_6$ satisfying $|W_6|=2\epsilon n$ and~\eqref{item:cond1} 
as there is at most one vertex of degree less than $2$.
Let $W_0:=\bigcup_{w\in V_6} N_H(w)$. Clearly, $|W_0|=4\epsilon n$ as $W_6$ is $3$-independent.

Next, we consider $W_i$ for $1\leq i\leq 5$. Let $H^2$ be the graph on the vertex set $W$ in which two vertices are adjacent if and only if two vertices are of distance at most $2$ in $H$. Since $H$ has the maximum degree $2$, the maximum degree of $H^2$ is at most $4$.
Using Hajnal--Szemer\'edi Theorem~\cite{Hajnal1970}, we may partition $W$ into $5$ independent sets of $H^2$ so that each part is of size at least $n/5-1$. By removing all vertices in $W_0\cup W_6$ from each part, $W_1, W_2,..., W_5$ may be obtained. Then, it is clear that each $W_i$ is $2$-independent in $H$ and
$|W_i|$ is at least $n/5-1-6\epsilon n \geq 2\epsilon n$, for  $i=1,...,5$. 
\end{proof}

A bijection from $W$ to $V=V(G)$ is called an \emph{embedding} of $H$  to $G$ if it maps each edge of $H$ to an edge of $G$.
We now find an embedding of $H$ to $G$ using an algorithm modified from the embedding algorithm in~\cite{Dellamonica13:universal, Dellamonica12:universal}: Take a partition $W_0,..., W_6$ of $W$ as described in Lemma~\ref{clm:partition of X}. We will embed $W_i$ into $V_0\cup\cdots\cup V_i$.

\renewcommand{\eps}{\epsilon}
\newcommand{\old}[1]{}
To map  $W_0$ into $V_0$, 
recall that $W_0 = N_H (W_6)$ and $|N_H ( w)|=2$ for all $w\in W_6$. For a matching ${\mathcal M}
=\{e_{_1},..., e_{_{2\eps n}}\}$ in $G[V_0]$
with (P1) and  $W_6=\{w_{_i},..., w_{_{2\eps n}}\}$, it is enough for us to  take a bijection, say $f_0$,  from $W_0$ to $V({\mathcal M})$ such that 
$N_H (w_{_i})$ is mapped to $e_{_i}$, where $V({\mathcal M})$ is the set of end vertices of all edges in ${\mathcal M}$.

The mapping $f_0$ is an embedding of $W_0$ to $V_0$: 
Since $V_6$ is a $3$-independent set in $H$, the sets $N_H(w)$, $w\in V_6$, are pairwise disjoint and there is no edge between them. Hence, every edge $e$ in $W_0$ belongs to $N_H(w)$ for some $w\in W_6$. As every $N_H(w)$ is mapped to an edge in $\cM\subset E(G[V_0])$ under $f_0$, the edge $e$ is mapped to an edge in $G[V_0]$.

Assuming  an embedding 
$$f_{i-1}:W_0\cup W_1\cup \cdots \cup W_{i-1}\rightarrow V_0\cup V_1\cup \cdots \cup V_{i-1}$$ is defined, $i=1,2,...,6$, we will embed $W_i$ into 
$V^*_i:=V_i \cup (V_0\cup V_1\cup \cdots \cup V_{i-1})\setminus {\rm Image}(f_{i-1}) $ 
to extend 
$f_{i-1}$ to an embedding $$f_i: W_0\cup W_1\cup \cdots \cup W_{i}\rightarrow V_0\cup V_1\cup \cdots \cup V_{i}. $$ 
Let $B_i(W_i, V^*_i)$, or just $B_i$, be   the bipartite graph in which $w\in W_i$ and $v\in V^*_i$ are adjacent 
if and only if 
$$f_{i-1}\Big(N_H(w)\cap(W_0\cup\cdots\cup W_{i-1})\Big)\subset N_G(v).$$ (See Figure~\ref{fig:B_i}).
\begin{figure}
\begin{center}
\includegraphics[scale=0.15]{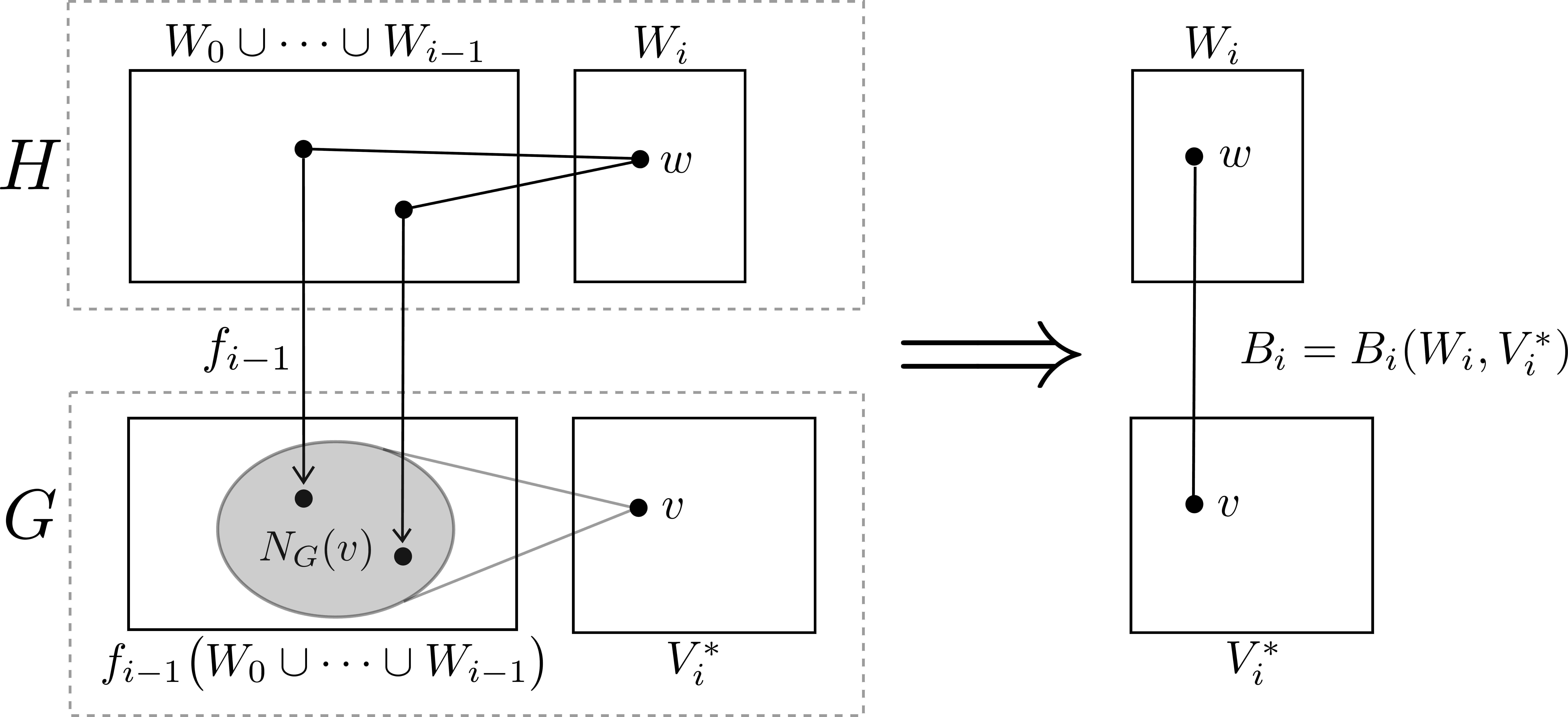}
\end{center}
\caption{The bipartite graph $B_i=B_i(W_i,V_i^*)$}
\label{fig:B_i}
\end{figure}
Or equivalently, for $L_i(w):=f_{i-1}\Big(N_H(w)\cap(W_0\cup\cdots \cup W_{i-1})\Big)$, $\cL_i:=\Big\{ L_i(w) : w\in W_i \Big\}$ and the bipartite graph $B(\cL_i, V^*_i)$ defined just before Definition~\ref{def:good},
 \begin{equation}\label{eq:two bipartite}\mbox{$w\sim v$ in $B_i(W_i,V_i^*)$ \hskip 0.5em if and only if \hskip 0.5em $L_i(w)\sim v$ in $B(\cL_i,V_i^*)$}.\end{equation}
   If possible, take  a
$W_{\! i}\, $-matching of $B_i$, i.e., a matching in $B_i$  that covers all vertices in $W_i$. (Later, we will show that this is possible).
 The image of $w \in W_i$ under  the mapping $f_i $  is defined to be the vertex in $V^*_i$  that is matched to $w$ in the $W_i$-matching. For $w\not\in W_i$, $f_i(w)=f_{i-1} (w)$. 

The mapping $f_i$ is an embedding of $W_i$ to $V^*_i$:
For an edge $e$ in $W_0\cup \cdots\cup W_i$, at most one end point of $e$ is in $W_i$ since $W_i$ is 2-independent, especially independent. If both ends of $e$ are in $W_0\cup \cdots \cup W_{i-1}$, then $f_i(e)=f_{i-1}(e)$ is an edge in $G$. If one end, say $w$, of $e$ is in $W_i$, then the other end, say $w'$, is in $N_H(w)\cap (W_0\cup \cdots \cup W_{i-1})$. Hence, $w\sim f_i(w)$ in $B_i(W_i,V_i^*)$ implies that $$f_{i-1}\Big(N_H(w)\cap(W_0\cup\cdots\cup W_{i-1})\Big)\subset N_G(f_i(w)),$$
in particular, $f_i(w')=f_{i-1}(w')\in N_G(f_i(w))$, i.e., $\{f_i(w'),f_i(w)\}$ is an edge.

It remains to show that there exists a $W_i$-matching in $B_i=B_i(W_i,V_i^*)$ for $i=1,..., 6$. We first show the following, which guarantees Hall's condition for a subset $U$ of $W_i$ satisfying some condition. 

\begin{lemma}\label{clm:matching} Let $i=1,..., 6$. If $U\subset W_i$ satisfying $|U|\leq |V^*_i|-n/C$, then
\begin{equation*}
|N_{B_i}(U)|\geq |U|.
\end{equation*}
\end{lemma}

\begin{proof}

 Let $U=U_0\cup U_1\cup U_2$, where $$U_j:=\Big\{w\in W_i : |N_H(w)\cap (W_0\cup \cdots \cup W_{i-1})|=j\Big\}.$$

\noindent If $U_0\neq \emptyset$, then $N_{B_i}(U)=V^*_i$ and hence $|N_{B_i}(U)|=|V^*_i|>|U|$ as $|U|\leq |V^*_i|-n/C$. We now assume that $U_0= \emptyset$.
Take $U_k$ such that $|U_k|\geq |U|/2$. 

\vskip 1em
\noindent \textbf{Case 1}: the case when $|U_k|\leq \frac{\delta}{C^k}\big(\frac{n}{\log n}\big)^{k/2}$. For $L_i(u):=f_{i-1}\Big(N_H(u)\cap(W_1\cup\cdots\cup W_{i-1})\Big)$ and $\cL(U_k)=\Big\{L_i(u)
: u\in U_k\Big\}$, We have that 
$$N_{_{B(\cL_i, V^*_i)}}\big(\cL(U_k)\big)\cap V_i\subset N_{B_i}(U_k):$$
For $v \in V_i$ with $L(u)\sim v$ in $B(\cL_i, V^*_i)$ for some $u\in U_k$, it follows from~\eqref{eq:two bipartite} that $u\sim v$ in $B_i$ for $u\in U_k$,
or equivalently, $v\in N_{B_i}(U_k)$. Notice that
$$N_{_{B(\cL_i, V^*_i)}}\big(\cL(U_k)\big)\cap V_i = N_{_{B(\cL(U_k),V_i)}}\Big(\cL(U_k)\Big).
$$
Property (P2) implies that
\begin{eqnarray}\label{eq:Case1}
 |N_{B_i}(U_k)|&\geq& |N_{_{B(\cL(U_k),V_i)}}(\cL(U_k))|  \geq    (1-\delta)C^k\epsilon n\Big(\frac{\log n}{n}\Big)^{k/2}|U_k|\nonumber \\ &\geq&  \frac{\epsilon C^2(\log n)}{2} |U_k|,
 \end{eqnarray}
 as $k=1\mbox{ or }2$.
 In particular, $|N_{B_i}(U_k)|\geq 2|U_k|\geq |U|$.
 
 \vskip 1em
\noindent\textbf{Case 2}: the case when $\frac{\delta}{C^k}\big(\frac{n}{\log n}\big)^{k/2}<|U_k|\leq \frac{\log n}{C^{k-1}}\big(\frac{n}{\log n}\big)^{k/2}$.
      Taking a subset $U'_k$ of $U_k$ of size $\frac{\delta}{C^k}\big(\frac{n}{\log n}\big)^{k/2}$, it follows from~\eqref{eq:Case1} that
\begin{eqnarray*}
|N_{B_i}(U_k)|&\geq &|N_{B_i}(U'_k)|\geq \frac{\epsilon\delta C^{2-k}(\log n)}{2} \Big(\frac{n}{\log n}\Big)^{k/2} \\
&=& \frac{\epsilon \delta C}{2}\cdot \frac{\log n}{C^{k-1}}\Big(\frac{n}{\log n}\Big)^{k/2}\geq 2|U_k|\geq |U|
\end{eqnarray*}
  as $C$ is sufficiently large and $\epsilon$ and $\delta$ are absolute constants.
      
      \vskip 1em
\noindent \textbf{Case 3}: the case when $|U_k|> \frac{\log n}{C^{k-1}}\big(\frac{n}{\log n}\big)^{k/2}$. We will show that 
\begin{equation*}
|N_{B_i}(U_k)|\geq |V^*_i|-n/C \hskip 0.5em (\geq |U|).
\end{equation*}
We first observe that there is no edge of $B_i=B_i(W_i,V^*_i)$ between $U_k$ and $V^*_i\setminus N_{B_i}(U_k)$. Hence, for $\cL(U_k)$ defined as in Case 1, it follows from~\eqref{eq:two bipartite}
that there is no edge of $B(\cL_i, V^*_i)$ between
$\cL(U_k)$ and $V^*_i\setminus N_{B_i}(U_k)$. This means that the induced subgraph $B\Big(\cL(U_k),V^*_i\setminus N_{B_i}(U_k)\Big)$ of $B(\cL_i, V^*_i)$ has no edge. 
Since $|\cL(U_k)|=|U_k|> \frac{\log n}{C^{k-1}}\big(\frac{n}{\log n}\big)^{k/2}$, the property (P2) yields that $$|V^*_i\setminus N_{B_i}(U_k)|<\frac{\log n}{C^{k-1}}\Big(\frac{n}{\log n}\Big)^{k/2}\leq \frac{n}{C},$$
which is equivalent to 
$|N_{B_i}(U_k)|\geq |V^*_i|-n/C$ as desired.
\end{proof}

\begin{corollary}\label{lem:matching B_i} For $i=1,..., 5$, there exists a $W_i$-matching in $B_i(W_i,V_i^*)$.
\end{corollary}

\begin{proof} One can easily show that $|W_i|<|V_i^*|-n/C$ for $i=1,..., 5$. Indeed, we have
\begin{eqnarray*}
|V^*_i| &=& |V_0\cup\cdots\cup V_i|-|W_0\cup\cdots \cup W_{i-1}| \\
&=& |W_i\cup\cdots\cup W_6|-|V_{i+1}\cup \cdots \cup V_6|,
\end{eqnarray*}
and
\begin{eqnarray*}
|V^*_i|-|W_i| 
&=& |W_{i+1}\cup\cdots\cup W_6|-|V_{i+1}\cup \cdots \cup V_6| \\
&\geq& (6-i)2\epsilon n -(6-i)\epsilon n =(6-i)\epsilon n 
\geq \epsilon n > \frac{n}{C}
\end{eqnarray*}
where the first inequality follows from~\eqref{eq:size of W_i} and the last inequality holds for a sufficiently large constant $C$. 
Clearly, for all $U\subset W_i$, we have
$|U|< |V^*_i|-n/C$ for $1\leq i\leq 5$.
Hence, Lemma~\ref{clm:matching} yields that for every $U\subset W_i$, we have $|N_{B_i}(U)|\geq |U|$. Consequently, Hall's theorem implies Corollary~\ref{lem:matching B_i}.
\end{proof}

Next, we consider the case when $i=6$.

\begin{lemma}
There exists a $W_6$-matching in $B_6=B_6(W_6,V^*_6)$.
\end{lemma}

\begin{proof} It suffices to check Hall's condition, that is, for every $U\subset W_6$,
\begin{equation}\label{eq:Hall}
|N_{B_6}(U)|\geq |U|.
\end{equation}
If $|U|\leq |V^*_6|-n/C = 2\epsilon n-n/C$, then Lemma~\ref{clm:matching} implies~\eqref{eq:Hall}. Hence, we assume that 
\begin{equation*}|U|> 2\epsilon n-n/C. \end{equation*}

Notice that $$|U|+\Big|N_{B_6}\Big(V^*_6\setminus N_{B_6}(U)\Big)\Big|\leq |W_6| = 2\epsilon n$$
since $U$ and $N_{B_6}\Big(V^*_6\setminus N_{B_6}(U)\Big)$ are disjoint. If $|V^*_6\setminus N_{B_6}(U)|\geq \frac{\delta n}{C^2 \log n}$, take a subset $Y$ of $V^*_6\setminus N_{B_6}(U)$ with $|Y|=\frac{\delta n}{C^2 \log n}$. Then, by equation~\eqref{eq:two bipartite} and Property (P1), we infer
\begin{equation*}
\Big|N_{B_6}\Big(V^*_6\setminus N_{B_6}(U)\Big)\Big|\geq |N_{B_6}(Y)| \geq \frac{\epsilon \delta}{8}n>\frac{n}{C}
\end{equation*}
and $|U|+\Big|N_{B_6}\Big(V^*_6\setminus N_{B_6}(U)\Big)\Big|>2\epsilon n,$ which is a contradiction. Therefore, $|V^*_6\setminus N_{B_6}(U)|< \frac{\delta n}{C^2 \log n}$. Then, Property (P1) together with~\eqref{eq:two bipartite} implies that 
$$
\Big|N_{B_6}\Big(V^*_6\setminus N_{B_6}(U)\Big)\Big|\geq \frac{C^2\epsilon \log n}{8}\big|V^*_6\setminus N_{B_6}(U)\big| > \big|V^*_6\setminus N_{B_6}(U)\big|.
$$
Since $|N_{B_6}(U)|+|V^*_6\setminus N_{B_6}(U)|=|V^*_6|=2\epsilon n $ and $|U|+\Big|N_{B_6}\Big(V^*_6\setminus N_{B_6}(U)\Big)\Big|\leq |W_6|=2\epsilon n$, we have
\begin{eqnarray*}
|N_{B_6}(U)|+|V^*_6\setminus N_{B_6}(U)| &\geq&  |U|+\Big|N_{B_6}\Big(V^*_6\setminus N_{B_6}(U)\Big)\Big| \\
&> & |U|+ \big|V^*_6\setminus N_{B_6}(U)\big|,
\end{eqnarray*}
that is, $|N_{B_6}(U)|>|U|$.
\end{proof}

\section{Random graph is good.}\label{sec:good}

In order to show Lemma~\ref{lem:good}, we need to show that there exists a positive constant $C$ such that the random graph $G(n,p)$ with $p=C\big(\frac{\log n}{n}\big)^{1/2}$ a.a.s. satisfies Properties (P1) and (P2) in Definition~\ref{def:good}. Our proof of Properties (P1) and (P2) of $G(n,p)$ will be given in Sections~\ref{sec:P1} and~\ref{sec:P2}, respectively.
In the proofs, we will use the following version of Chernoff's bound.

\begin{lemma}[\textbf{Chernoff's bound}, Corollary 4.6 in~\cite{MU2005}] \label{lem:Chernoff} Let $X_i$ be independent random variables such that $\Pr[X_i=1]=p_i$ and $\Pr[X_i=0]=1-p_i$, and let $X=\sum_{i=1}^{n} X_i$.  
For $0<\lambda<1$,
\begin{equation*}\Pr \Big[|X-\EE(X)|\geq \lambda\EE(X)\Big]\leq 2 \exp\Big(-\frac{\lambda^2}{3}\EE(X)\Big). \end{equation*}
\end{lemma}

\subsection{Property~(P1)}\label{sec:P1} 
In order to show that $G(n,p)$ with $p=C\big(\frac{\log n}{n}\big)^{1/2}$ a.a.s. satisfies Property~(P1), it suffices to show the following lemma.

 \begin{lemma}\label{lem:P1} There exists a positive constant $C$ such that $G(n,p)$ with $p=C\big(\frac{\log n}{n}\big)^{1/2}$ a.a.s. satisfies the following:
 There exists  a matching $\cM$ with $|\cM|=2\epsilon n$ in the subgraph of $G(n,p)$ induced on $V_0$  such that for all $U\subset V\setminus V(\cM)$ with $ |U|\leq  \frac{\delta n}{C^2\log n}$, 
we have $$\Big|\Big\{ \{a,b\}\in \cM \; \big| \; a\sim u, b\sim u \mbox{ for some } u\in U\Big\}\Big|\geq \frac{C^2\log n}{16n} |\cM| |U|.$$

\end{lemma}

\begin{proof} Let $G_1=G\big(n-6\epsilon n, \frac{2\log n}{n}\big)$ on the vertex set $V_0$ and $G_2=G(n,p/2)$ on the vertex set $V$. It is easy to see that $G(n,p)$ on $V$ stochastically contains $G_1\cup G_2$. Hence, it is enough to show that $G_1\cup G_2$ a.a.s. contains a matching $\cM$ described in Lemma~\ref{lem:P1}.

 The result of Erd\H{o}s and R\'enyi~\cite{Erdos1966} implies that $G_1$ a.a.s. contains a  matching in $V_0$ covering all but at most one vertex. Hence, $G_1$ on $V_0$ a.a.s. contains a matching of size $2\epsilon n$. Take such a matching $\cM$ in $G_1$.
 
 Let
  \begin{eqnarray*}X(U):=\Big|\Big\{ e\in \cM \; \big| \; e\subset N_{_{G_2}}(u) \mbox{ for some } u\in U\Big\}\Big|.\end{eqnarray*}
 Notice that $X(U)$ is the sum of independent and identically distributed (i.i.d.) random variables $X_e$, $e\in \cM$, where
$$\displaystyle X_e=\left\{ \begin{matrix} \hskip 0.4em 1 \hskip 1em \mbox{ if $e\subset N_{_{G_2}}(u) \mbox{ for some } u\in U$} \\  \phantom{} \hskip -8em 0  \hskip 1em \mbox{ otherwise.} \end{matrix} \right.$$

Since $|U|\leq \frac{\delta n}{C^2\log n}$, we have that for each $e\in \cM$,
\begin{eqnarray*}
\Pr\Big[X_e=0\Big]&=& \Big(1-\Big(\frac{p}{2}\Big)^2\Big)^{|U|}\leq 1-\frac{|U|p^2}{4}+\frac{|U|^2p^4}{32} \\
&\leq & 1-\Big(1-\frac{\delta}{8}\Big)\frac{|U|p^2}{4} \leq 1-\frac{|U|p^2}{8},
\end{eqnarray*}
or equivalently, 
$\Pr\big[X_e=1\big]\geq\frac{1}{8}|U|p^2.$ Hence, $$\EE\big(X(U)\big)\geq \frac{p^2}{8}|\cM||U|=\frac{C^2\log n}{8n}|\cM||U|.$$
Chernoff's bound (Lemma~\ref{lem:Chernoff})  yields that
\begin{eqnarray*}
\Pr\Big[X(U)<\frac{C^2\log n}{16n}|\cM||U|\Big]&\leq & \Pr\Big[|X(U)-\EE\big(X(U)\big)|\geq \frac{1}{2}\EE\big(X(U)\big)\Big] \\
&\leq & 2\exp\Big(-0.01p^2|\cM||U|\Big) \\
&=& 2\exp\Big(-0.01\cdot 2\epsilon C^2 |U|\log n \Big) \\
&\leq & 2\exp\big(-3|U|\log n\big)\leq \frac{2}{n^{3|U|}}.
\end{eqnarray*}
Therefore, we infer
\begin{eqnarray*}
&&\hskip -5em \Pr\Big[\exists\; U\in V\setminus V(\cM) \mbox{ with } |U|\leq \frac{\delta n}{C^2\log n} \mbox{ such that } X(U)<\frac{C^2\log n}{16n}|\cM||U|\Big] \\
&\leq & \sum_{\ell=1}^{n}n^{\ell}\frac{2}{n^{3\ell}}\leq n\cdot \frac{2}{n^2}=o(1),
\end{eqnarray*}
which completes the proof of~Lemma~\ref{lem:P1}.
\end{proof}

\subsection{Property~(P2)}\label{sec:P2} We now  show that $G(n,p)$ with $p=C\big(\frac{\log n}{n}\big)^{1/2}$  a.a.s. satisfies Property~(P2). First, recall the following definition which was given just before Definition~\ref{def:good}: For a graph $G$ on $V$ and $k=1 \mbox{ or }2$, let $U\subset V$ and $\cL$ be a collection of pairwise disjoint $k$-subsets of $V\setminus U$. We consider a bipartite graph $B(\cL, U)$ between $\cL$ and $U$, in which $L\in\cL$ and $u\in U$ are adjacent if and only if $L\subset N_G(u)$.

In order to show that $G(n,p)$ with $p=C\big(\frac{\log n}{n}\big)^{1/2}$ a.a.s. satisfies Property~(P2), it suffices to show the following lemma.

\begin{lemma}\label{lem:P2}
There exists a positive constant $C$ such that $G(n,p)$ with $p=C\big(\frac{\log n}{n}\big)^{1/2}$ a.a.s. satisfies the following: Let $k=1 \mbox{ or } 2$, and $\cL$ be a collection of pairwise disjoint $k$-subsets of $V$. 

\begin{enumerate}[$(a)$]
\item\label{item:P2_small} 
 If $\displaystyle |\cL|\leq \frac{\delta}{C^k}\Big(\frac{n}{\log n}\Big)^{k/2}$, then, for $V_i$ with $V_i\cap\Big(\bigcup_{L\in \cL}L\Big)=\emptyset$, $i=1,..., 6$, we have that
\begin{equation*}
|N_{_{\! B(\cL,V_i)}}(\cL)|\geq (1-\delta)C^k\Big(\frac{\log n}{n}\Big)^{k/2}|\cL| |V_i|.
\end{equation*}

\item\label{item:P2_large}
 If $\displaystyle |\cL|\geq \frac{\log n}{C^{k-1}}\Big(\frac{n}{\log n}\Big)^{k/2}$, then, for all $U$ with $\displaystyle |U|\geq \frac{\log n}{C^{k-1}}\Big(\frac{n}{\log n}\Big)^{k/2}$ and $U\cap\Big(\bigcup_{L\in \cL}L\Big)=\emptyset$, 
the graph $B(\cL,U)$ has at least one edge.
\end{enumerate}
 
 \end{lemma}

\begin{proof} For a proof of~\eqref{item:P2_small} of Lemma~\ref{lem:P2}, we observe that $X(\cL,V_i):=|N_{_{B(\cL,V_i)}}(\cL)|$ is the sum of i.i.d. random variables $X_v$, $v\in V_i$,  where $$\displaystyle X_v=\left\{ \begin{matrix} \hskip 0.4em 1 \hskip 1em \mbox{ if  $L \subset N_G(v)$  for some $L\in \cL$} \\  \phantom{} \hskip -7.8em 0  \hskip 1em \mbox{ otherwise.} \end{matrix} \right. $$
Since $|L|=k$ for all $L\in \cL$ and $p^k|\cL|\leq \delta=0.01$, we have
 \begin{equation*}\label{eq:E(X)(2)}\EE\Big(X(\cL,V_i)\Big)=|V_i|\Big(1-(1-p^k)^{|\cL|}\Big)\geq \(1-\delta/2\) p^k{|\cL|}|V_i|=\(1-\delta/2\) C^k\Big(\frac{\log n}{n}\Big)^{k/2}|\cL||V_i|.\end{equation*} 
Chernoff's bound (Lemma~\ref{lem:Chernoff}) implies that
\begin{eqnarray*}\Pr\Big[X(\cL,V_i)< (1- \delta)C^k\Big(\frac{\log n}{n}\Big)^{k/2}|\cL||V_i| \Big] &\leq &  \Pr\Big[|X(\cL,V_i)-\EE\Big(X(\cL,V_i)\Big)| \geq  \frac{\delta}{2} \EE\Big(X(\cL,V_i)\Big)\Big] \\ &\leq & 2 \exp\Big(-\frac{\delta^2}{12}\EE\Big(X(\cL,V_i)\Big)\Big),\end{eqnarray*}
and by $p^k|V_i|\geq p^2\epsilon n=C^2\epsilon \log n$,
\begin{eqnarray*}
\Pr\Big[X(\cL,V_i)< (1- \delta)C^k\Big(\frac{\log n}{n}\Big)^{k/2}{|\cL|}|V_i| \Big] \leq 2\exp\big( -3{|\cL|}\log n\big) 
=  \frac{2}{n^{3|\cL|}}.
\end{eqnarray*}
Therefore, we infer that
\begin{eqnarray*}
&& \hskip -5em \Pr\Big[\exists\; V_i, \cL \mbox{ with $1\leq {|\cL|}\leq \frac{\delta}{C^k}\big(\frac{n}{\log n}\big)^{k/2}$ such that } X(\cL,V_i)< (1- \delta)C^k\Big(\frac{\log n}{n}\Big)^{k/2}{|\cL|}|V_i| \Big] \\ &\leq& 6  
\sum_{\ell=1}^{n}n^{\ell} \frac{2}{n^{3\ell}}  \leq  6n\frac{2}{n^2}=o(1),
\end{eqnarray*}
which completes the proof of~\eqref{item:P2_small} of Lemma~\ref{lem:P2}.

For a proof of~\eqref{item:P2_large} of Lemma~\ref{lem:P2},
we observe that the number $Y(\cL,U)$ of edges in $B(\cL,U)$ is the sum of i.i.d. random variables $Y_{L,u}$ for $L\in \cL$ and $u\in U$, where
$$\displaystyle Y_{L,u}=\left\{ \begin{matrix} \hskip 0.4em 1 \hskip 1em \mbox{ if  $L \subset N_G(u)$} \\  \phantom{} \hskip -0.8em 0  \hskip 1em \mbox{ otherwise.} \end{matrix} \right. $$
Since $|L|=k$ for all $L\in \cL$, we have
 $\EE\Big(Y(\cL,U)\Big)=p^k{|\cL|}{|U|}.$
    Chernoff's bound (Lemma~\ref{lem:Chernoff})  yields that 
\begin{eqnarray*}
\Pr\Big[Y(\cL,U)=0\Big]
&\leq &\Pr\Big[\big|Y(\cL,U)-\EE\big(Y(\cL,U)\big)\big|\geq \frac{1}{2} \EE\big(Y(\cL,U)\big)\Big] \\ &\leq &2\exp\Big(-\frac{1}{12}\EE\big(Y(\cL,U)\big)\Big) 
\leq  2\exp\Big(-\frac{1}{12}p^k{|\cL|}{|U|}\Big).
\end{eqnarray*}
For $\ell\geq \frac{\log n}{C^{k-1}}\big(\frac{n}{\log n}\big)^{k/2}$ and $r\geq \frac{\log n}{C^{k-1}}\big(\frac{n}{\log n}\big)^{k/2}$, the number of $\cL$ with $|\cL|=\ell$ is at most $\binom{n}{k}^{\ell} \leq n^{k\ell}$ and the number of $U$ with $|U|=r$ is at most $\binom{n}{r}$, and we have
 \begin{eqnarray*}
 && \hskip -5em \Pr\Big[\exists\; \cL, U \mbox{ with ${|\cL|=\ell}, {|U|=r}$ such that } Y(\cL,U)=0\Big] \\
&&\leq  n^{k{\ell}} n^{r} \cdot 2\exp\Big(-\frac{1}{12}p^k{\ell}{r}\Big) \leq     2\exp\Big(\big(k{\ell}+{r}\big)\log n-\frac{1}{12}p^k{\ell}{r}\Big).
 \end{eqnarray*}
Since $p^k \ell =C^k\big(\frac{\log n}{n}\big)^{k/2}\ell\geq C\log n$ and $p^kr=C^k\big(\frac{\log n}{n}\big)^{k/2}r\geq C\log n$, we have that
\begin{eqnarray*}
(k\ell+r)\log n \leq 0.01C(\ell+r)\log n \leq 0.01\big(p^k\ell r+p^k\ell r\big) \leq 0.02p^k{\ell}{r},
\end{eqnarray*}
and hence,
$$ n^{k{\ell}} n^{r} \cdot 2\exp\Big(-\frac{1}{12}p^k{\ell}{r}\Big)\leq 2\exp\Big(-\frac{1}{24}p^k{\ell}{r}\Big)\leq 2\exp\Big(-\mbox{$\frac{C^{2-k}}{24}\big(\frac{n}{\log n}\big)^{k/2}(\log n)^2$} \Big)\leq 2\exp\big(-n^{1/2}\big). $$
Therefore, we infer that
\begin{eqnarray*}
&& \hskip -5em \Pr\Big[\exists\; \cL, U \mbox{ with ${|\cL|}, {|U|}\geq \frac{\log n}{C^{k-1}}\big(\frac{n}{\log n}\big)^{k/2}$ such that } Y(\cL,U)=0\Big] \\
&\leq & \frac{n}{k}\cdot n \cdot 2\exp\big(-n^{1/2}\big)  =o(1),
\end{eqnarray*}
which completes the proof of~\eqref{item:P2_large} of Lemma~\ref{lem:P2}.
\end{proof}

\section{Concluding remarks}

One may ask about how the approach of this paper can
be used for the case that $d\geq 3$.
We believe that our approach for finding a suitable matching given in Lemma~\ref{lem:P1} can be also applied in order to find a suitable family of vertex disjoint $d$-cliques when $d\geq 3$ and $p\geq C\big(\frac{\log n}{n}\big)^{1/d}$.
 This approach together with an embedding algorithm modified from the algorithm in Dellamonica, Kohayakawa, \Rodl~and \Rucinski~\cite{Dellamonica13:universal, Dellamonica12:universal} may provide a  simpler proof of Theorem~\ref{thm:universality}. 

As a further research direction, we are interested in resolving the following problem.
\begin{problem}
For an integer $d\geq 2$, determine the largest constant $a=a(d)$ with $0\leq a\leq 1$ such that if $p\geq n^{-a+o(1)}$, then $G(n,p)$ is a.a.s. $\cF_n(d)$-universal.
\end{problem}

An easy observation mentioned in the introduction gives an upper bound $\frac{2}{d+1}$ for $a$. The current best lower bound is $\frac{1}{d}$  based on  the result in Dellamonica, Kohayakawa, \Rodl~and \Rucinski~\cite{Dellamonica13:universal, Dellamonica12:universal} and this paper.

\old{
Let $V_0$ be a vertex set of size at least $\frac{3}{4}n$.

 \begin{lemma}\label{lem:P1_d} For $d\geq 3$, there exists a positive constant $C=C(d)$ such that $G(n,p)$ with $p=C\big(\frac{\log n}{n}\big)^{1/d}$ a.a.s. satisfies the following:
 There exists  a family $\cK$ of disjoint $d$-cliques with $|\cK|=2\epsilon n$ in the subgraph of $G(n,p)$ induced on $V_0$  such that for all $U\subset V\setminus V(\cK)$ with $ |U|\leq  \frac{\delta n}{(4C)^d\log n}$, 
we have $$\Big|\Big\{ K\in \cK \; \big| \; K\subset N_{_{G(n,p)}}(u) \mbox{ for some } u\in U\Big\}\Big|\geq \frac{C^d\log n}{2\cdot 5^d n} |\cK| |U|.$$

\end{lemma}

\begin{proof} Let $G_1=G\big(\frac{3}{4}n, \frac{2\log n}{n}\big)$ on the subset of $V_0$ of size $\frac{3}{4}n$, and $G_2=G(n,p/2)$ on the vertex set $V$. It is easy to see that $G(n,p)$ on $V$ stochastically contains $G_1\cup G_2$. Hence, it is enough to show that $G_1\cup G_2$ a.a.s. contains a matching $\cK$ described in Lemma~\ref{lem:P1_d}.

 The result by Johansson, Kahn and Vu~\cite{Johansson2008}  implies that $G_1$ a.a.s. contains a $K_d$ factor covering all but at most $d-1$ vertices. Hence, $G_1$ a.a.s. contains a family of $2\epsilon n$ disjoint $d$-cliques. Take such a family $\cK$ in $G_1$.
 
 Let
  \begin{eqnarray*}X(U):=\Big|\Big\{ K\in \cK \; \big| \; K\subset N_{_{G_2}}(u) \mbox{ for some } u\in U\Big\}\Big|.\end{eqnarray*}
 Notice that $X(U)$ is the sum of i.i.d. random variables $X_K$, $K\in \cK$, where
$$\displaystyle X_K=\left\{ \begin{matrix} \hskip 0.4em 1 \hskip 1em \mbox{ if $K\subset N_{_{G_2}}(u) \mbox{ for some } u\in U$} \\  \phantom{} \hskip -8.4em 0  \hskip 1em \mbox{ otherwise.} \end{matrix} \right.$$

Since $|U|\leq \frac{\delta n}{(4C)^d\log n}$, we have that for each $K\in \cK$,
\begin{eqnarray*}
\Pr\Big[X_K=0\Big]&=& \Big(1-\Big(\frac{p}{2}\Big)^d\Big)^{|U|}\leq 1-\frac{|U|p^d}{2^d}+\frac{|U|^2p^{2d}}{2^{2d+1}}  \leq 1-\frac{|U|p^d}{2^{d+1}},
\end{eqnarray*}
or equivalently, 
$\Pr\big[X_e=1\big]\geq\frac{1}{2^{d+1}}|U|p^d.$ Hence, $$\EE\big(X(U)\big)\geq \frac{p^d}{2^{d+1}}|\cK||U|=\frac{C^d\log n}{2^{d+1}n}|\cK||U|.$$
Chernoff's bound (Lemma~\ref{lem:Chernoff})  yields that
\begin{eqnarray*}
\Pr\Big[X(U)<\frac{C^d\log n}{2^{d+2}n}|\cK||U|\Big]&\leq & \Pr\Big[|X(U)-\EE\big(X(U)\big)|\geq \frac{1}{2}\EE\big(X(U)\big)\Big] \\
&\leq & 2\exp\Big(-0.01\frac{C^d \log n}{2^{d+1} n}|\cK||U|\Big) \\
&=& 2\exp\Big(-\frac{0.01 \epsilon C^d }{2^{d} } |U|\log n \Big) \\
&\leq & 2\exp\big(-3|U|\log n\big)\leq \frac{2}{n^{3|U|}}.
\end{eqnarray*}
Therefore, we infer
\begin{eqnarray*}
&&\hskip -5em \Pr\Big[\exists\; U\in V\setminus V(\cK) \mbox{ with } |U|\leq \frac{\delta n}{(4C)^d\log n} \mbox{ such that } X(U)<\frac{C^d\log n}{2^{d+2}n}|\cK||U|\Big] \\
&\leq & \sum_{\ell=1}^{n}n^{\ell}\frac{2}{n^{3\ell}}\leq n\cdot \frac{2}{n^2}=o(1),
\end{eqnarray*}
which implies the proof of~Lemma~\ref{lem:P1_d}.
\end{proof}

}

\begin{acknowledge} 
 The second author thanks D. Dellamonica, Y. Kohayakawa, V. R\"odl, and A. Ruci\'nski 
 for their helpful comments.
\end{acknowledge}

\providecommand{\bysame}{\leavevmode\hbox to3em{\hrulefill}\thinspace}
\providecommand{\MR}{\relax\ifhmode\unskip\space\fi MR }
\providecommand{\MRhref}[2]{%
  \href{http://www.ams.org/mathscinet-getitem?mr=#1}{#2}
}
\providecommand{\href}[2]{#2}
\def\MR#1{\relax}


\begin{thebibliography}{10}

\bibitem{Alon2002}
N.~Alon and V.~Asodi, \emph{Sparse universal graphs}, J. Comput. Appl. Math.
  \textbf{142} (2002), no.~1, 1--11, Probabilistic methods in combinatorics and
  combinatorial optimization. \MR{1910514 (2004d:05102)}

\bibitem{Alon2007-2}
N.~Alon and M.~Capalbo, \emph{Sparse universal graphs for bounded-degree
  graphs}, Random Structures Algorithms \textbf{31} (2007), no.~2, 123--133.
  \MR{2343715 (2008e:05104)}

\bibitem{Alon08}
\bysame, \emph{Optimal universal graphs with deterministic embedding},
  Proceedings of the {N}ineteenth {A}nnual {ACM}-{SIAM} {S}ymposium on
  {D}iscrete {A}lgorithms (New York), ACM, 2008, pp.~373--378. \MR{2485323}

\bibitem{Alon2007-3}
N.~Alon, M.~Capalbo, Y.~Kohayakawa, V.~R{\"o}dl, A.~Ruci{\'n}ski, and
  E.~Szemer{\'e}di, \emph{Universality and tolerance (extended abstract)}, 41st
  {A}nnual {S}ymposium on {F}oundations of {C}omputer {S}cience ({R}edondo
  {B}each, {CA}, 2000), IEEE Comput. Soc. Press, Los Alamitos, CA, 2000,
  pp.~14--21. \MR{1931800}

\bibitem{Alon01}
\bysame, \emph{Near-optimum universal graphs for graphs with bounded degrees
  (extended abstract)}, Approximation, randomization, and combinatorial
  optimization ({B}erkeley, {CA}, 2001), Lecture Notes in Comput. Sci., vol.
  2129, Springer, Berlin, 2001, pp.~170--180. \MR{1910361}

\bibitem{Alon2007}
N.~Alon, M.~Krivelevich, and B.~Sudakov, \emph{Embedding nearly-spanning
  bounded degree trees}, Combinatorica \textbf{27} (2007), no.~6, 629--644.
  \MR{2384408 (2009d:05110)}

\bibitem{Babai1982}
L.~Babai, F.~R.~K. Chung, P.~Erd{\H{o}}s, R.~L. Graham, and J.~H. Spencer,
  \emph{On graphs which contain all sparse graphs}, Theory and practice of
  combinatorics, North-Holland Math. Stud., vol.~60, North-Holland, Amsterdam,
  1982, pp.~21--26. \MR{806964 (86m:05057)}

\bibitem{Balogh2010}
J.~Balogh, B.~Csaba, M.~Pei, and W.~Samotij, \emph{Large bounded degree trees
  in expanding graphs}, Electron. J. Combin. \textbf{17} (2010), no.~1,
  Research Paper 6, 9. \MR{2578901 (2011d:05322)}

\bibitem{Balogh2011}
J.~Balogh, B.~Csaba, and W.~Samotij, \emph{Local resilience of almost spanning
  trees in random graphs}, Random Structures Algorithms \textbf{38} (2011),
  no.~1-2, 121--139. \MR{2768886 (2012d:05340)}

\bibitem{Bhatt1989}
S.~N. Bhatt, F.~R.~K. Chung, F.~T. Leighton, and A.~L. Rosenberg,
  \emph{Universal graphs for bounded-degree trees and planar graphs}, SIAM J.
  Discrete Math. \textbf{2} (1989), no.~2, 145--155. \MR{990447 (90b:05071)}

\bibitem{Bollobas2001}
B.~Bollob{\'a}s, \emph{Random graphs}, second ed., Cambridge Studies in
  Advanced Mathematics, vol.~73, Cambridge University Press, Cambridge, 2001.
  \MR{1864966 (2002j:05132)}

\bibitem{Capalbo1999}
M.~R. Capalbo and S.~R. Kosaraju, \emph{Small universal graphs}, Annual {ACM}
  {S}ymposium on {T}heory of {C}omputing ({A}tlanta, {GA}, 1999), ACM, New
  York, 1999, pp.~741--749 (electronic). \MR{1798099 (2001i:05139)}

\bibitem{Chung1978}
F.~R.~K. Chung and R.~L. Graham, \emph{On graphs which contain all small
  trees}, J. Combinatorial Theory Ser. B \textbf{24} (1978), no.~1, 14--23.
  \MR{0505812 (58 \#21808a)}

\bibitem{Chung1979}
\bysame, \emph{On universal graphs}, Second {I}nternational {C}onference on
  {C}ombinatorial {M}athematics ({N}ew {Y}ork, 1978), Ann. New York Acad. Sci.,
  vol. 319, New York Acad. Sci., New York, 1979, pp.~136--140. \MR{556017
  (81c:05028)}

\bibitem{Chung1983}
\bysame, \emph{On universal graphs for spanning trees}, J. London Math. Soc.
  (2) \textbf{27} (1983), no.~2, 203--211. \MR{692525 (84i:05064)}

\bibitem{Chung1978-2}
F.~R.~K. Chung, R.~L. Graham, and N.~Pippenger, \emph{On graphs which contain
  all small trees. {II}}, Combinatorics ({P}roc. {F}ifth {H}ungarian {C}olloq.,
  {K}eszthely, 1976), {V}ol. {I}, North-Holland, Amsterdam, 1978, pp.~213--223.
  Colloq. Math. Soc. J\'anos Bolyai, 18. \MR{0505813 (58 \#21808b)}

\bibitem{Dellamonica13:universal}
D.~Dellamonica, K.~Kohayakawa, V.~R\"odl, and A.~Ruci\'{n}ski, \emph{An
  improved upper bound on the density of universal random graphs}, Submitted.

\bibitem{Dellamonica12:universal}
D.~Dellamonica, K.~Kohayakawa, and A.~Ruci\'{n}ski, \emph{An improved upper
  bound on the density of universal random graphs (extended abstract)}, Latin
  American Theoretical Computer Science Symposium, LATIN, 2012, pp.~231--242.

\bibitem{Dellamonica2008_2}
D.~Dellamonica, Jr. and Y.~Kohayakawa, \emph{An algorithmic
  {F}riedman-{P}ippenger theorem on tree embeddings and applications},
  Electron. J. Combin. \textbf{15} (2008), no.~1, Research Paper 127, 14.
  \MR{2448877 (2009i:05059)}

\bibitem{Dellamonica2008}
D.~Dellamonica, Jr., Y.~Kohayakawa, V.~R{\"o}dl, and A.~Ruci{\'n}ski,
  \emph{Universality of random graphs}, Proceedings of the {N}ineteenth
  {A}nnual {ACM}-{SIAM} {S}ymposium on {D}iscrete {A}lgorithms (New York), ACM,
  2008, pp.~782--788. \MR{2487648}

\bibitem{Dellamonica2012}
\bysame, \emph{Universality of random graphs}, SIAM J. Discrete Math.
  \textbf{26} (2012), no.~1, 353--374. \MR{2902650}

\bibitem{Erdos1947}
P.~Erd{\"o}s, \emph{Some remarks on the theory of graphs}, Bull. Amer. Math.
  Soc. \textbf{53} (1947), 292--294. \MR{0019911 (8,479d)}

\bibitem{Erdos1966}
P.~Erd{\H{o}}s and A.~R{\'e}nyi, \emph{On the existence of a factor of degree
  one of a connected random graph}, Acta Math. Acad. Sci. Hungar. \textbf{17}
  (1966), 359--368. \MR{0200186 (34 \#85)}

\bibitem{Friedman1987}
J.~Friedman and N.~Pippenger, \emph{Expanding graphs contain all small trees},
  Combinatorica \textbf{7} (1987), no.~1, 71--76. \MR{905153 (88k:05063)}

\bibitem{Hajnal1970}
A.~Hajnal and E.~Szemer{\'e}di, \emph{Proof of a conjecture of {P}. {E}rd{\H
  o}s}, Combinatorial theory and its applications, {II} ({P}roc. {C}olloq.,
  {B}alatonf\"ured, 1969), North-Holland, Amsterdam, 1970, pp.~601--623.
  \MR{0297607 (45 \#6661)}

\bibitem{Janson00}
S.~Janson, T.~{\L}uczak, and A.~Rucinski, \emph{Random graphs},
  Wiley-Interscience Series in Discrete Mathematics and Optimization,
  Wiley-Interscience, New York, 2000. \MR{1782847 (2001k:05180)}

\bibitem{Johannsen2013}
D.~Johannsen, M.~Krivelevich, and W.~Samotij, \emph{Expanders are universal for
  the class of all spanning trees}, Combin. Probab. Comput. \textbf{22} (2013),
  no.~2, 253--281. \MR{3021334}

\bibitem{Johansson2008}
A.~Johansson, J.~Kahn, and V.~Vu, \emph{Factors in random graphs}, Random
  Structures Algorithms \textbf{33} (2008), no.~1, 1--28. \MR{2428975
  (2009f:05243)}

\bibitem{Komlos1983}
J.~Koml{\'o}s and E.~Szemer{\'e}di, \emph{Limit distribution for the existence
  of {H}amiltonian cycles in a random graph}, Discrete Math. \textbf{43}
  (1983), no.~1, 55--63. \MR{680304 (85g:05124)}

\bibitem{Korshunov1977}
A.~Korshunov, \emph{A solution of a problem of {P}. {E}rd{\H{o}}s and {A}.
  {R}{\'e}nyi about hamilton cycles in non-oriented graphs}, Metody Diskr.
  Anal. Teoriy Upr. Syst., Sb. Trudov Novosibirsk. \textbf{31} (1977), 17--56
  (in Russian).

\bibitem{MU2005}
M.~Mitzenmacher and E.~Upfal, \emph{Probability and computing}, Cambridge
  University Press, Cambridge, 2005, Randomized algorithms and probabilistic
  analysis. \MR{2144605 (2006d:68002)}

\bibitem{Rodl1981}
V.~R{\"o}dl, \emph{A note on universal graphs}, Ars Combin. \textbf{11} (1981),
  225--229. \MR{629872 (83d:05060)}

\end{thebibliography}
\end{document}